\documentclass[10pt
]{article}

\usepackage{amsmath,amssymb} \usepackage{latexsym}
\usepackage{theorem}
\usepackage{graphicx}
\usepackage{rotate}
\usepackage{pinlabel}
\usepackage{mathrsfs}

\DeclareSymbolFont{AMSb}{U}{msb}{m}{n}
\DeclareSymbolFontAlphabet{\mathbb}{AMSb} 

\newcommand{\E}{E}
\newcommand{\F}{F}

\newtheorem{theorem}{Theorem}[section]
\newtheorem{proposition}[theorem]{Proposition}
\newtheorem{lemma}[theorem]{Lemma}
\newtheorem{corollary}[theorem]{Corollary}

{{\theorembodyfont{\rmfamily}

\newtheorem{notation}[theorem]{Notation}
\newtheorem{definition}[theorem]{Definition}
\newtheorem{remark}[theorem]{Remark}
\newtheorem{example}[theorem]{Example}

}

\author{
Peter Pfaffelhuber and Heinz Weisshaupt\\
\\ZBSA University of Freiburg\\
Habsburgerstrasse 49\\
79104 Freiburg, Germany
\\
heinz.weisshaupt@zbsa.uni-freiburg.de
\\
peter.pfaffelhuber@stochastik.uni-freiburg.de}

\title{Sensitivity analysis of one parameter semigroups exemplified by the Wright--Fisher diffusion}

\begin{document}

\maketitle

\begin{abstract}
We consider the sensitivity, with respect to a parameter $\theta$, of parametric families of operators $A_{\theta}$, vectors $\pi_{\theta}$ corresponding to the adjoints $A_{\theta}^{*}$ of $A_{\theta}$ via $A_{\theta}^{*}\pi_{\theta}=0$ and one parameter semigroups $t\mapsto e^{tA_{\theta}}$. We display formulas relating weak differentiability of $\theta\mapsto \pi_{\theta}$ (at $\theta=0$) to weak differentiability of $\theta\mapsto A_{\theta}^{*}\pi_{0}$ and $[e^{A_{\theta}t}]^{*}\pi_{0}$. We give two applications: The first one concerns the sensitivity of the Ornstein--Uhlenbeck 
process with respect to its location parameter. The second one provides new insights regarding the Wright--Fisher diffusion for small mutation parameter.
  
\end{abstract}



\section*{Introduction}

Sensitivities of parametric families of dynamical systems (with respect to the parameter) have been studied in the context of stochastic processes \cite{Pflug96} as well as partial differential equations \cite{Bues} and 
are useful tools in optimization and control. We consider sensitivities in the setting of one parameter semigroups. This setting constitutes a unifying approach to continuous time Markov processes \cite{Bob} and linear PDEs \cite{EngelNagel}, \cite{Bob}. In particular it allows an elegant treatment of the sensitivities of the Wright--Fisher diffusion.

We consider one parameter semigroups\footnote{Mappings $U$ from $[0,\infty)$ to the space of linear operators $U(t):E\to E$ on some linear space $E$ such that $U(t+s)=U(t)U(s)$} $U:=(U(t))_{t\in [0,\infty)}$, that can be represented as $U(t)=e^{tA}$ for some linear operator $A:E\to E$ on some linear space $E$. This does not exclude generators $A$ that are usually considered as 'unbounded', since we do in general not suppose that $E$ is a Banach space. (For information on the  Hille-Yosida Generation Theorem providing $U(t)=e^{tA}$ for unbounded $A$ in a Banach space setting consult \cite[Chapter II, Section 3]{EngelNagel}.)   

Since the topic of this article is not the behaviour of a single semigruop, but the sensitivity of their behaviour with respect to small perturbations of a parameter $\theta$, we do not only consider one generator $A$ or one semigroup $U$, but consider parametric families $(A_{\theta})_{\theta\in\Theta}$ and $(U_{\theta})_{\theta\in\Theta}$ of generators $A_{\theta}$ and semigroups $U_{\theta}=(e^{tA_{\theta}})_{t\in [0,\infty)}$, with $0\in\Theta\subseteq {\mathbb R}$. (Note that $\theta$ is an additional parameter and must not be confused with the parameter $t$ of a single one parameter semigroup.)

We further consider a second linear space $F$ that is in duality with $E$ and parametric families $(\pi_{\theta})_{\theta\in\Theta}$ such that $\pi_{\theta}\in F$ and $A_{\theta}^{*}\pi_{\theta}=0$ for an adjoint (dual) $A_{\theta}^{*}:F\to F$ of $A_{\theta}$.

We first show in Lemma \ref{A-fundamental-diff-relation-lemma} that differentiability of $\theta\mapsto A_{0}^{*}\pi_{\theta}$ at $\theta=0$ is equivalent to differentiability of $\theta\mapsto A_{\theta}^{*}\pi_{0}$ at $\theta=0$ provided that 
\begin{equation*}
[A_{0}^{*}-A_{\theta}^{*}][\pi_{\theta}-\pi_{0}]\to 0\ \hbox{ for }\ \theta\to 0,
\end{equation*}
with all limits taken in the weak sense with respect to the given duality. Further 
\begin{equation}\label{equality-of-partial}
\frac{\partial A_{0}^{*}\pi_{\theta}}{d\theta}|_{\theta=0}=-\frac{\partial A_{\theta}^{*}\pi_{0}}{d\theta}|_{\theta=0},
\end{equation}
with the involved limits again taken in the weak sense.

Thus Lemma \ref{A-fundamental-diff-relation-lemma} makes the calculation of $\frac{\partial A_{0}^{*}\pi_{\theta}}{d\theta}|_{\theta=0}$ in some cases easier than the calculation of $\frac{\partial \pi_{\theta}}{\partial \theta}|_{\theta=0}$. This fact is exemplified by Remark \ref{d_pi_theta-FW}, Proposition \ref{A'_0-pi_0-FW} and Remark \ref{|nu>=partial} in the case that $A_{\theta}$ is the generator of the Wright--Fisher diffusions.

Next we prove Theorem \ref{main-result}, the main result of the article. In the case that $F$ is a subspace of the algebraic dual $E'$ of $E$, the theorem provides a formula (see Remark \ref{main-cor}) for $\frac{\partial [U_{\theta}(t)]^{*}}{\partial \theta}\pi_{0}|_{\theta=0}$ with $[U_{\theta}(t)]^{*}$ denoting the uniquely determined adjoint of $U_{\theta}(t)$, i.e., a formula for the sensitivity of $t\mapsto [U_{\theta}(t)]^{*}\pi_{0}$ with respect to the parameter $\theta$. This formula involves $\nu:=\frac{\partial A_{\theta}^{*}\pi_{0}}{d\theta}|_{\theta=0}$ and operators $V_{0}(t)$ given by the series expansion $V_{0}(t)=\sum_{i=1}^{\infty} \frac{t^{n}}{n!}A^{n-1}$.

It is a simple fact that in the case that $F\subseteq E'$ we have $[U_{\theta}(t)]^{*}\pi_{\theta}=\pi_{\theta}$, i.e., $\pi_{\theta}$ is a stationary vector of $[U_{\theta}(t)]^{*}$.

One of the hypotheses of Theorem \ref{main-result} is that $E=\bigcup_{j\in J} E_{j}$ with $E_{j}$ Banach spaces with respect to the norms $\Vert . \Vert_{j}$, such that $A_{\theta}(E_{j})\subseteq E_{j}$ for all $\theta\in\Theta$ and the restriction $A_{\theta}|_{E_{j}}$ of $A_{\theta}$ to $E_{j}$ is bounded with respect to $\Vert. \Vert_{j}$. It is essential for the proof of Theorem \ref{main-result} that the $E_{j}$ do not depend on $\theta$. Note that this hypothesis is fairly restrictive. It prevents us, for example, from the investigation of diffusion equations with arbitrary coefficient functions. It allows however quite interesting insights in the following situation:

Our abstract results are applicable to the case that $E$ equals the space of polynomials on some appropriate real interval and the generators $A_{\theta}$ are of the form
\begin{equation}\label{diffop-A}
A_{\theta}:=\sum_{i=1}^{n} p_{i}(x)\cdot q_{i}(\theta)\frac{\partial^{i}}{{\partial x}^{i}}
\end{equation}
with $p_{i}$ polynomials of degree less than $i$ and $q_{i}$ differentiable functions. This is due to the fact that the operators $A_{\theta}$ leave for any $k\in {\mathbb N}$ the spaces of polynomials of degree less than $k$ invariant. We apply our results to two examples of diffusions operators $A_{\theta}$ that fulfill (\ref{diffop-A}), i.e. to differential operators of the form (\ref{diffop-A}) with $n=2$.

In the first example we demonstrate the applicability of our results to  a parametric family of Ornstein--Uhlenbeck semigroups $t\mapsto [U_{\theta}(t)]^{*}$ corresponding to generators
\begin{equation*}
A_{\theta}:=(\theta-x)\frac{\partial }{\partial x}+\frac{\sigma^2}{2}\frac{\partial^{2} }{{\partial x}^{2}},
\end{equation*}
considering without loss of generality the case $\sigma^{2}=1$.

OU-semigroups and stochastic processes corresponding to these semigroups are frequently used in interest rate modeling. The parameter $\theta$ is interpreted as the interest rate to which the process reverts. (Compare with \cite[Vol 2, Chapter 46]{Wilmott} and \cite[ Section 9.3]{Steele}.) In this example all derivatives can be represented by functions. Further it is possible to calculate the derivatives directly since the evolution of the OU-semigroup is explicitly given by (\ref{evolution-of-OU}). Thus the example of the OU-semigroup is just of an illustrative nature that does not really rely on the developed theory. This is quite different for our second example:

In our second example we consider a parametric family of Wright--Fisher diffusions with mutation and without selection, that can be described by the semigroups $[U_{\theta}(t)]^{*}$ corresponding to the generators
\begin{equation}\label{generator-wright-fisher-diffusion-intro}
A_{\theta}:=(1-x)\theta\frac{\partial }{\partial x} + x \kappa \frac{\partial}{\partial x}+x(1-x)\frac{1}{2}\frac{\partial^2}{\partial x^2}.
\end{equation}
Wright--Fisher diffusions are useful tools in population genetics, describing the distributions of allele-frequencies in a population (see \cite{Ewens} and Remark \ref{interpretation-of-WF}). Note that the stationary distribution $\pi_{0}$ of $[U_{0}(t)]^{*}$ is in the degenerate case $\theta=0$ and $\mu>0$ given by the Dirac measure at $0$. We calculate the sensitivity $\frac{d\pi_{\theta}}{d\theta}|_{\theta=0}$ of $\pi_{\theta}$ at $\theta=0$ as well as the sensitivities $A_{0}^{*}\frac{d\pi_{\theta}}{d\theta}$ and $\frac{\partial U_{\theta}^{*}\pi_{0}}{d\theta}|_{\theta=0}$. We obtain (Proposition \ref{|nu>=partial}, formula (\ref{=partial})) that $A^{*}\frac{d\pi_{\theta}}{d\theta}|_{\theta=0} = \frac{\partial}{\partial x}$, i.e., the sensitivities under consideration are in general not representable by functions or measures on $[0,1]$, but are general linear functionals on the space of polynomials on $[0,1]$.

What makes the concrete calculation of these sensitivities difficult is the fact that the involved operators are not diagonalizable. It is however possible to construct a basis (see Remark \ref{basis-remark}) of the space of polynomials such that $A_{0}$ is almost diagonalizable in the sense that equation (\ref{A^ĸ_psi_n}) holds. This enables us to provide a relatively simple recursive formula for the sensitivity $\frac{\partial U_{\theta}}{\partial \theta} \pi_{0}|_{\theta=0}$ (Theorem \ref{wf-rec-theorem}) with respect to this basis.

Although our methods are purely functional analytic and non-probabilistic in nature we can---in the case of diffusion processes---interpret the action of the sensitivities on the space of polynomials in a probabilistic manner: The action of the derivative on the $n$-th monomial is simply the derivative of the $n$-th moment of the parametric family of probability measures under consideration. This can be further interpreted in the case of the Wright--Fisher diffusion (compare with Remark \ref{interpretation-of-WF}).

Derivatives of diffusion semigruops with respect to an additional parameter $\theta$ have been dealt with in the context of mathematical finance mainly in the context of the stochastic calculus of variations, but also in a PDE context. For an introduction to such results consult \cite{Mall-Thal} (especially \cite[ Theorems 2.2 and 2.3]{Mall-Thal}) and \cite{Wilmott}. (For an elementary approach to the relationship of diffusion processes and diffusion equations consult \cite{WeLind}). Derivatives of Markov kernels have been considered in \cite{Pflug96} and \cite{HeiHoWe}. An extension to derivatives of general operators in a Banach space context, relating the derivatives of the operators to the derivatives of their stationary vectors has been given in \cite{WeStationarySensitivity}.      

\section{Generators and stationary vectors}

\begin{remark}
We follow in the style of presentation of our general functional analytic results \cite[Sections 16 and 21]{KeNa}.
\end{remark}

\begin{definition}
We say that $(\E, \F, \langle . | .\rangle)$ is a dual pairing of the linear spaces $\E$ and $\F$ if $\langle . | . \rangle :\E\times \F \to \mathbb R$ is bilinear. We denote by $w(\E,\F)$ and $w(\F,\E)$ the weak topologies induced by the families of mappings $\{ \xi\mapsto \langle \xi | \mu\rangle \mid {\mu\in\F} \}$ and $\{ \mu \mapsto \langle \xi | \mu\rangle \mid {\xi\in\E} \}$, respectively. We say that the dual pairing is separating if $w(\E,\F)$ and $w(\F, \E)$ are Hausdorff. In the case of a separating dual pairing we may identify $\F$ with a subspace of the algebraic dual $E'$ and vice versa $\E$ with a subspace of $\F'$. We denote the spaces of $w(\E,\F)$-continuous linear transformations $A:\E\to\E$ by ${\cal L}_{w}(\E)$ and the space of $w(\F,\E)$-continuous linear transformation $A:\F\to\F$ by ${\cal L}_{w}(\F)$, respectively. We say that the linear transformations $S:\E\to \E$ and $T:\F\to \F$ are dual if for arbitrary $\xi\in\E$ and $\mu\in\F$ we have $\langle S\xi | \mu \rangle = \langle \xi | T\mu\rangle$. In the case that the pairing is separating, we call a dual transformations an adjoint and note that the adjoint is uniquely determined. 
\end{definition}

\begin{remark}\label{partial-notation}
Let $(\E, \F, \langle . | .\rangle)$ be a dual pairing. Let $E'$ denote the algebraic dual of $E$, i.e, the space of all linear functionals (continuous or not) on $E$. Given a parametric family $(\mu_{\theta})_{\theta\in\Theta}\in F^{\Theta}$ such that 
\begin{equation}
\forall \xi\in E\ \ \ \lim_{\theta\to 0} \theta^ {-1}\langle \xi | \mu_{\theta} - \mu_{0}\rangle\ \ \hbox{ exists,}
\end{equation}
we let $\frac{\partial \mu_{\theta}}{\partial \theta}|_{\theta=0}\in E'$ denote the unique linear functional such that
\begin{equation}
\lim_{\theta\to 0} \theta^ {-1}\langle \xi | \mu_{\theta} - \mu_{0}\rangle = \frac{\partial \mu_{\theta}}{\partial \theta}|_{\theta=0}(\xi).
\end{equation}
We call $\frac{\partial \mu_{\theta}}{\partial \theta}|_{\theta=0}$ the $E$-derivative of $\theta\mapsto \mu_{\theta}$ at $\theta=0$ and say that some $\nu\in F$ represents $\frac{\partial \mu_{\theta}}{\partial \theta}|_{\theta=0}$ if
\begin{equation}
\langle \xi | \nu \rangle = \frac{\partial \mu_{\theta}}{\partial \theta}|_{\theta=0}(\xi).
\end{equation}
In the case that the dual pairing is separating the representative $\nu\in F$ is unique (if it exists). 

\end{remark}

\begin{proposition}\label{continuity-of-adjoint}
A mapping $S:\E\to \E$ possesses a dual $T:\F\to\F$ if and only if $S\in {\cal L}_{w}(\E)$. Further $T\in {\cal L}_{w}(\F)$ (since $T$ possesses the dual $S$). 
\end{proposition}
\begin{proof}
See \cite{KeNa} 21.1.
\end{proof}


\begin{lemma}\label{A-fundamental-diff-relation-lemma}
Let $(\E, \F, \langle . | . \rangle)$ be a dual pairing. Let $0\in \Theta\subseteq\mathbb R$ with $0$ an accumulation point of $\Theta$. For each $\theta\in\Theta$ let $A_{\theta}\in {\cal L}_{w}(\E)$ and $\pi_{\theta}\in \F$. Denote by $A_{\theta}^{*}$ a dual (the adjoint) of $A_{\theta}$. (Note that the existence of $A_{\theta}^{*}$ is granted by Proposition \ref{continuity-of-adjoint}).
Suppose that:
\begin{equation}\label{A-Lpi=0}
(\forall\theta\in\Theta)\ \ \ A_{\theta}^{*}\pi_{\theta}=0,
\end{equation}
and
\begin{equation}\label{A-prod-quicker-than-h}
(\forall\xi\in\E)\ \ \ \lim_{\theta\to 0} \theta^{-1}\langle \xi | [A_{\theta}-A_{0}]^{*} (\pi_{\theta}-\pi_{0})\rangle = 0.
\end{equation}
Then $\theta\mapsto A_{0}^{*}\pi_{\theta}$ possesses an $\E$-derivative at $0$ that is represented by $-\nu\in \F$, i.e., 
\begin{equation}\label{A-Psi-diff}
(\forall \xi\in\E)\ \ \
\lim_{\theta \to 0} \theta^{-1}\langle\xi | A^{*}_{0}(\pi_{\theta}-\pi_{0})\rangle=\langle\xi | -\nu \rangle,
\end{equation}
if and only if $\theta\mapsto A_{\theta}^{*}\pi_{0}$ possesses an $\E$-derivative at $0$ represented by $\nu\in F$, i.e., 
\begin{equation}\label{A-=-L0pi0'}
(\forall \xi\in\E)\ \ \
\lim_{\theta\to 0} \theta^{-1}\langle\xi | [A^{*}_{\theta}-A^{*}_{0}]\pi_{0}\rangle=\langle\xi | \nu \rangle.
\end{equation}  
\end{lemma}

\begin{proof} Let $\xi\in\E$ be arbitrary. Calculation gives:
\begin{equation}\label{A-0=.+.+.}
\begin{split}
\left \langle\xi \left | \frac{A_{\theta}^{*}\pi_{\theta}-A_{0}^{*}\pi_{0}}{\theta}\right. \right \rangle - \left \langle\xi\left |
\frac{[A_{\theta}^{*}-A_{0}^{*}][\pi_{\theta}-\pi_{0}]}{\theta}\right . \right \rangle\\
= \left \langle\xi\left | \frac{[A_{\theta}^{*}-A_{0}^{*}]\pi_{0}}{\theta}\right. \right \rangle + \left \langle\xi\left | \frac{A_{0}^{*}[\pi_{\theta}-\pi_{0}]}{\theta}\right. \right \rangle
\end{split}
\end{equation}
The limit $\theta\to 0$ on the left hand side of equation (\ref{A-0=.+.+.}) exists and equals $0$ by (\ref{A-Lpi=0}) and (\ref{A-prod-quicker-than-h}). Thus the same is true for the right hand side, i.e., 
\begin{equation*}
\lim_{\theta\to 0}\left (\left \langle\xi\left | \frac{[A_{\theta}^{*}-A_{0}^{*}]\pi_{0}}{\theta}\right. \right \rangle + \left \langle\xi\left | \frac{A_{0}^{*}[\pi_{\theta}-\pi_{0}]}{\theta}\right. \right \rangle \right )=0
\end{equation*}
and thus further that
\begin{equation}\label{limit=limit-if-exists}
\lim_{\theta\to 0} \left \langle\xi\left | \frac{[A_{\theta}^{*}-A_{0}^{*}]\pi_{0}}{\theta}\right. \right \rangle = -\lim_{\theta\to 0} \left \langle\xi\left | \frac{A_{0}^{*}[\pi_{\theta}-\pi_{0}]}{\theta}\right. \right \rangle
\end{equation}
in the sense that if the limit on one side of equation (\ref{limit=limit-if-exists}) exists, then so does the limit on the other one. Since $\xi\in\E$ was arbitrarily chosen (\ref{limit=limit-if-exists}) establishes the equivalence of (\ref{A-Psi-diff}) and (\ref{A-=-L0pi0'}).
\end{proof}

\begin{remark}\label{remark-diff-by-cont}
If $\lim_{\theta\to 0} \frac{1}{\theta}\langle \xi | \pi_{\theta}-\pi_{0} \rangle = \langle \xi | \pi_{0}' \rangle$ for some $\pi_{0}'\in \F$, then (\ref{A-Psi-diff}) is fulfilled with $\nu=A_{0}^{*}\pi_{0}'$. This follows from the continuity of $A_{0}^{*}$ (Proposition \ref{continuity-of-adjoint}). 
\end{remark}

\begin{remark}
If the spaces $\E$ and $\F$ in Lemma \ref{A-fundamental-diff-relation-lemma} are additionally endowed with norms $\Vert . \Vert_{\E}$ and $\Vert . \Vert_{\F}$ respectively, such that $\langle . | . \rangle$ is continuous with respect to these norms then (\ref{A-prod-quicker-than-h}) is implied by 
\begin{equation}\label{norm-prod-quicker-than-h}
(\forall\xi\in\E)\ \ \lim_{\theta\to 0} \theta^{-1}\Vert [A_{\theta}-A_{0}] \Vert \cdot \Vert \pi_{\theta}-\pi_{0}\Vert = 0,
\end{equation}
that is further implied by the norm-Lipschitz continuity of $\theta\mapsto A_{\theta}$ and $\theta\mapsto \pi_{\theta}$ at $\theta=0$. (Compare with \cite{WeStationarySensitivity} Proof of Theorem 3.1)
\end{remark}


\section{Sensitivity analysis of semigroups}

\begin{remark}
Let $(H,\Vert .\Vert)$ denote a normed linear space and let ${\Vert .\Vert}$-$\lim_{n\to\infty} \xi_{n}$ denote the limit with respect to $\Vert .\Vert$ of the sequence $(\xi_{n})_{n\in {\mathbb N}}\in H^{\mathbb N}$. We say that a linear operator $A:H\to H$ is $\Vert.\Vert$-bounded if $\Vert A\Vert_{\cal L} := \sup_{\xi\in H\setminus \{ 0\} }\frac{\Vert A\xi\Vert}{\Vert \xi\Vert}<\infty$. We denote the space of $\Vert.\Vert$-bounded operators by ${\cal L}(H)$ and note that $({\cal L}(H),\Vert .\Vert_{\cal L})$ formes a normed algebra. If convenient we write $\Vert A\Vert$ instead of $\Vert A \Vert_{\cal L}$.
\end{remark}


\begin{lemma}\label{convergence-of-exp}
Suppose that $\E:=\bigcup_{j\in J}\E_{j}$ is a linear space and that $(\E_{j},\Vert .\Vert_{j})$ are (for $j\in J$) complete normed spaces. Suppose further that $A_{\theta}:\E\to\E$ is linear, $A_{\theta}(\E_{j})\subseteq \E_{j}$ and $A_{\theta}|_{\E_{j}}\in {\cal L}(\E_{j})$. Then for any $j\in J$ and any $\xi_{j}\in\E_{j}$
\begin{equation}\label{U-V-def}
\begin{split}
U_{\theta}(t)\xi_{j}:=e^{tA_{\theta}}:=\Vert .\Vert_{j}\hbox{-}\lim_{N\to\infty} \sum_{n=0}^{N} \frac{(tA_{\theta})^{n}}{n!}\xi_{j}\in E_{j},\\
V_{\theta}(t)\xi_{j} :=\Vert .\Vert_{j}\hbox{-}\lim_{N\to \infty} \sum_{n=1}^{N} \frac{t}{n} \frac{(tA_{\theta})^{n-1}}{(n-1)!}\xi_{j}\in E_{j}
\end{split}
\end{equation} 
exist, i.e. (\ref{U-V-def}) well-defines operators $U_{\theta}(t), V_{\theta}(t):\E\to \E$.
\end{lemma}
\begin{proof}
The Chauchy sequences 
\begin{equation}
\left (\sum_{n=0}^{N} \frac{(tA_{\theta})^{n}}{n!}\xi_{j}\right )_{N\in\mathbb N}\ \hbox{ and }\ \left (\sum_{n=1}^{N} \frac{t}{n} \frac{(tA_{\theta})^{n-1}}{(n-1)!}\xi_{j}\right )_{N\in\mathbb N}
\end{equation}
converge by completeness of $\E_{j}$ with respect to $\Vert .\Vert_{j}$.
\end{proof}
\begin{remark}\label{integral-representation-of-V}
Note that $V_{0}(t)\xi=\int_{s=0}^{t} U_{0}(s)\xi\, ds$, with the integral either taken in the sense of Riemann or Lebesgue.
\end{remark}

\begin{notation}
In the situation of Lemma \ref{convergence-of-exp} we write $\Vert A \Vert_{j}$ instead of $\Vert A|_{E_{j}}\Vert_{j}$.
\end{notation}

\begin{remark}
The following Theorem \ref{main-result} is the main result of this article. Its reformulation Corollary \ref{main-cor} is concerned with the sensitivity of $[U_{\theta}(t)]^{*}\pi_{0}$ with respect to the parameter $\theta$ at $\theta=0$. Note that $[U_{0}(t)]^{*}\pi_{0}=\pi_{0}$, i.e., $\pi_{0}$ is an equilibrium for the dynamics governed by $t\mapsto [U_{\theta}(t)]^{*}$ and thus Corollary \ref{main-cor} provides formulas for the calculation of the first order effect of small perturbations of the parameter $\theta$ to the systems dynamics at an equilibrium (of the unperturbed system). 
\end{remark}

\begin{theorem}\label{main-result}
Suppose that the hypotheses of Lemma \ref{A-fundamental-diff-relation-lemma} and the hypothesis of Lemma \ref{convergence-of-exp} are fulfilled. Suppose 
that for any $j\in J$ the mapping 
\begin{equation}
\theta\mapsto \Vert A_{\theta}-A_{0} \Vert_{j}
\end{equation}
is Lipschitz continuous at $\theta=0$ and that 
\begin{equation}\label{sup_Psi_j}
\sup_{\xi\in\E_{j}\setminus \{ 0 \}} \frac{\langle \xi | \pi_{0} \rangle}{\Vert \xi\Vert_{j}}< c_{j} < \infty,
\end{equation}
for appropriate constants $c_{j}$ i.e., $\xi\mapsto \langle \xi | \pi_{0}\rangle$ defines (for any $j\in J$) a $\Vert . \Vert_{j}$-continuous linear functional on $E_{j}$.
Let $U_{\theta}(t), V_{\theta}(t):\E\to\E$ denote the operators defined by (\ref{U-V-def}).
Then:\\ \\
(i)\ \ \ \ \ \ \parbox{10cm}{The mappings $\theta\mapsto \Vert U_{\theta}(t)-U_{0}(t) \Vert_{j}$ and $\theta\mapsto \Vert V_{\theta}(t)-V_{0}(t) \Vert_{j}$ are---for any $j\in J$---continuous at $\theta=0$,}\\ \\
(ii)\ \ \ \ \ \parbox{10cm}{$\forall \xi\in E\ \ \ \ \lim_{\theta\to 0} \theta^{-1}\langle [U_{\theta}(t)-U_{0}(t)] \xi | \pi_{0}\rangle = - \langle V_{0}(t)  \xi | \nu \rangle$  
}\\ \\
for any $\nu\in F$ fulfilling (\ref{A-=-L0pi0'}).
\end{theorem}

\begin{corollary}\label{main-cor}
In the case that $F=E'$ and $\langle \xi | \mu \rangle := \mu(\xi)$, we may reformulate conclusion (ii) of Theorem \ref{main-result} (using Remark \ref{partial-notation}) as 
\begin{equation}
\left \langle \xi \left | \frac{\partial [U_{\theta}(t)]^{*} \pi_{0}}{\partial \theta}  \right . \right \rangle = \frac{\partial }{\partial \theta} \langle \xi | [U_{\theta}(t)]^{*} \pi_{0} \rangle  = - \langle V_{0}(t)  \xi | \nu \rangle
\end{equation}
with $[U_{\theta}(t)]^{*}$ the adjoint of $U_{\theta}(t)$. 
\end{corollary}
\underline{Proof of Theorem \ref{main-result}:}
To prove (i) we just show conitnuity of $\theta\mapsto \Vert V_{\theta}(t)-V_{0}(t) \Vert_{j}$ at $\theta=0$, since continuity of $\theta\mapsto \Vert [U_{\theta}(t)-U_{0}(t) \Vert_{j}$ is proved completely analoguous. Let
\begin{equation}
V_{\theta,N}(t):=\sum_{n=1}^{N} \frac{t}{n} \frac{(tA_{\theta})^{n-1}}{(n-1)!}.
\end{equation}
Lipschitz continuity of $\theta\mapsto \Vert A_{\theta}-A_{0}\Vert_{j}$ at $\theta=0$ implies that $\theta\mapsto \Vert A_{\theta} \Vert_{j}<c<\infty$ for some appropriate $c>0$ in a $0$-neighborhood $(-\rho,\rho)\cap \Theta$, thus that 
\begin{equation}
\theta\mapsto \Vert A_{\theta}^n-A_{0}^{n}\Vert_{j} = \left \Vert \sum_{i=0}^{n-1} (A_{\theta})^{n-i-1}[A_{\theta}-A_{0}](A_{0})^{i} \right \Vert_{j} \leq n\cdot c^{n-1} \Vert A_{\theta}-A_{0}\Vert_{j}
\end{equation}
on $(-\rho,\rho)\cap \Theta$ and thus further that
\begin{equation}
\theta\mapsto \Vert V_{\theta,N}(t) - V_{0,N}(t) \Vert_{j} = \left \Vert \sum_{n=1}^{N} \frac{t^n}{n!} (A_{\theta}^{n-1}-A_{0}^{n-1}) \right \Vert_{j}
\end{equation}
is continuous at $\theta=0$.
Assertion (i) follows easily from the fact that the limit of a uniformly convergent sequence of function that are all continuous at $\theta=0$ is again continuous at $\theta=0$ and that $\theta\mapsto \Vert V_{\theta,N}(t) - V_{0,N})(t) \Vert_{j}$ converges uniformly on $(-\rho,\rho)$ to $\theta\mapsto \Vert V_{\theta}(t)-V_{0}(t) \Vert_{j}$.\\ \\
To prove (ii) note that $A_{0}^{*}\pi_{0}=0$ implies
\begin{equation}\label{A_pi_0}
[(A^{*}_{\theta})^n-(A^{*}_{0})^{n}]\pi_{0}=(A^{*}_{\theta})^{n-1}[A^{*}_{\theta}-A^{*}_{0}]\pi_{0}.
\end{equation}
Thus
\begin{equation}\label{diff-U-to-V}
\begin{split}
& \langle [U_{\theta}(t)-U_{0}(t)]\xi | \pi_{0}\rangle \stackrel{(a)}{=} \lim_{N\to\infty} \left \langle \sum_{n=1}^{N} \frac{t^{n}}{n!}[(A_{\theta})^n-(A_{0})^{n}] \xi | \pi_{0}\right \rangle\\
& = \lim_{N\to\infty}  \left \langle \xi | \sum_{n=1}^{N} \frac{t^{n}}{n!}[(A^{*}_{\theta})^n-(A^{*}_{0})^{n}] \pi_{0} \right \rangle
 \stackrel{(b)}{=} \lim_{N\to\infty} \left \langle \xi | \sum_{n=1}^{N} \frac{t}{n}\frac{(tA^{*}_{\theta})^{n-1}}{(n-1)!} [A^{*}_{\theta}-A^{*}_{0}] \pi_{0} \right \rangle\\
& \stackrel{}{=} \lim_{N\to\infty} \left \langle [A_{\theta}-A_{0}]\sum_{n=1}^{N} \frac{t}{n}\frac{(tA_{\theta})^{n-1}}{(n-1)!} \xi | \pi_{0} \right \rangle \stackrel{(c)}{=}
\langle  [A_{\theta}-A_{0}] V_{\theta}(t) \xi | \pi_{0}\rangle\\ 
& =\langle V_{\theta}(t) \xi | [A^{*}_{\theta}-A^{*}_{0}] \pi_{0}\rangle,
\end{split}
\end{equation}
with (a) and (c) consequences of the $\Vert .\Vert$-continuity of $\xi\mapsto \langle \xi | \pi_{0}\rangle$ (Hypotheisis  (\ref{sup_Psi_j})) and (b) a consequence of (\ref{A_pi_0}).

Let $\nu_{\theta}:=\theta^{-1}[A^{*}_{\theta}-A^{*}_{0}]\pi_{0}$. From (\ref{sup_Psi_j}) and by the Lipschitz continuity of $\theta\mapsto\Vert A_{\theta} - A_{0}\Vert$ at $\theta=0$ we obtain that $\exists \ell_{j}>0$ such that $\forall \theta\in\Theta\setminus \{ 0\}$ and $\forall \xi\in\E_{j}$ 
\begin{equation}\label{lipschitz->bounded}
\langle \xi | \nu_{\theta}\rangle =  \theta^{-1} \langle  [A_{\theta}-A_{0}] \xi | \pi_{0}\rangle \leq c_{j}\Vert [A_{\theta}-A_{0}] \xi \Vert\leq \ell_{j} \Vert \xi\Vert_{j}.
\end{equation}
From (\ref{lipschitz->bounded}) we get $\forall \theta\in\Theta\setminus \{ 0\}$ and $\forall \xi\in\E_{j}$ that
\begin{equation}\label{V-psi-bounded}
\langle [V_{\theta}(t)-V_{0}(t)]\xi | \nu_{\theta} \rangle \leq \ell_{j} [\Vert V_{\theta}(t) - V_{0}(t)]\xi \Vert_{j}.
\end{equation}
From (\ref{V-psi-bounded}) and (i) we obtain that
\begin{equation}\label{lim-V-psi=0}
\forall j\in J\ \forall \xi_{j}\in E_{j}\ \ \ \ \lim_{\theta\to 0} \langle [V_{\theta}(t)-V_{0}(t)]\xi | \nu_{\theta} \rangle \leq \ell_{j} \lim_{\theta\to 0} \Vert [V_{\theta}(t) - V_{0}(t)]\xi \Vert_{j} = 0.
\end{equation}
From the weak differentiability of $\theta\mapsto A_{\theta}^{*}\pi_{0}$ at $\theta=0$, i.e., from (\ref{A-=-L0pi0'}), we obtain that 
\begin{equation}\label{existence-by-weak-differentiability}
\forall\xi \in E\ \ \ \ \lim_{\theta\to 0} \langle V_{0}(t) \xi |  \nu_{\theta}-\nu \rangle=0
\end{equation}
Using (\ref{lim-V-psi=0}) and (\ref{existence-by-weak-differentiability}) we obtain for $\xi\in\E$, i.e. for $\xi\in\E_{j}$ for appropriate $j$, that
\begin{equation}\label{Vpi-theta-Vpi-0}
\begin{split}
\lim_{\theta\to 0}\langle V_{\theta}(t) \xi | \nu_{\theta} \rangle - \langle V_{0}(t) \xi | \nu \rangle = \lim_{\theta\to 0} \langle [V_{\theta}(t) - V_{0}(t)] \xi | \nu_{\theta} \rangle + \lim_{\theta\to 0}\langle V_{0}(t) \xi | [\nu_{\theta} -\nu]\rangle =0.
\end{split}
\end{equation}
We finally obtain using (\ref{diff-U-to-V}) and (\ref{Vpi-theta-Vpi-0}) that
\begin{equation}
\begin{split}
\lim_{\theta\to 0}\theta^{-1}\langle [U_{\theta}(t) - U_{0}(t)] \xi | \pi_{0}\rangle
\stackrel{(\ref{diff-U-to-V})}{=}  \lim_{\theta\to 0}\theta^{-1}\langle V_{\theta}(t) \xi |  [A^{*}_{\theta}-A^{*}_{0}] \pi_{0}\rangle\\
 = \lim_{\theta\to 0}\langle V_{\theta}(t) \xi | \nu_{\theta} \rangle \stackrel{(\ref{Vpi-theta-Vpi-0})}{=}\langle V_{0}(t) \xi | \nu \rangle,
\end{split}
\end{equation} 
Thus (ii) has been proved.


\section{The OU-semigroup}\label{section-OU}
 
\begin{definition}
Let $I$ be some interval in ${\mathbb R}$. Denote by ${\cal R}(I)$ the vector space of all polynomial functions $\xi:I\to {\mathbb R}$. Denote by ${\cal R}(I)'$ the space of all linear functionals with values in ${\mathbb R}$ on ${\cal R}(I)$ and let $\langle . | . \rangle$ denote the natural dual pairing between ${\cal R}(I)$ and ${\cal R}(I)'$. Note that any such functional is uniquely determined on the space of monomials and thus any such functional may be uniquely represented as $\F(\xi)=(\frac{\sum_{n=1}^{\infty}\partial^{n}}{\partial x^{n}}|_{x=0}\xi(x))$.
\end{definition}

\begin{example}
Let $A_{\theta}:{\cal R}({\mathbb R})\to {\cal R}({\mathbb R})$ be given by
\begin{equation}
A_{\theta}:=(\theta-x)\frac{\partial }{\partial x}+\frac{1}{2}\frac{\partial^{2} }{{\partial x}^{2}}.
\end{equation}
and let $A_{\theta}^{*}:{\cal R}({\mathbb R})'\to {\cal R}({\mathbb R})'$ denote the dual of $A_{\theta}$. Let $\pi_{\theta}\in {\cal R}({\mathbb R})'$ be implicitly given by
\begin{equation*}
\forall\xi\in {\cal R}({\mathbb R})\ \ \ \ \langle \xi | \pi_{\theta} \rangle = \int_{-\infty}^{\infty} \xi(x) \frac{1}{\sqrt{\pi}}e^{(x-\theta)^{2}}\; dx.
\end{equation*}
Define the Ornstein--Uhlenbeck semigroup $t\mapsto [U_{\theta}(t)]^{*}$ as the semigroup of the adjoints $[U_{\theta}(t)]^{*}$ of the operators $U_{\theta}(t)=e^{tA_{\theta}}$. Then $A_{\theta}^{*}\pi_{\theta}=0$ and the action of the  Ornstein--Uhlenbeck semigroup on $\pi_{0}$ is given by 
\begin{equation}\label{evolution-of-OU}
\langle U_{\theta} \xi | \pi_{0} \rangle = \langle \xi | U^{*}_{\theta}\pi_{0}\rangle = \int_{-\infty}^{\infty} \xi(x) \frac{1}{\sqrt{\pi}}e^{(x-(1-e^{-t})\theta)^{2}}\; dx.
\end{equation}
Further $\frac{1}{\sqrt{\pi}}e^{(x-\theta)^{2}}$ and $\frac{1}{\sqrt{\pi}}e^{(x-(1-e^{-t})\theta)^{2}}$ are the densities of the normal distributions $N(\theta,1/2)$ and $N((1-e^{-t})\theta,1/2)$, respectively. Thus by Theorem \ref{main-result} (ii) and (\ref{evolution-of-OU})
\begin{equation*}
\begin{split}
& \langle \xi | [V_{0}(t)]^{*}\nu \rangle = \langle V_{0}(t)\xi | \nu\rangle = \lim_{\theta\to 0} \theta^{-1}\langle [U_{\theta}(t)-U_{0}(t)] \xi | \pi_{0}\rangle\\
& = \lim_{\theta\to 0} \theta^{-1}\int_{-\infty}^{\infty} \xi(x) \frac{1}{\sqrt{\pi}}\left ( e^{(x-(1-e^{-t})\theta)^{2}} - e^{x^{2}} \right )\; dx\\
& = \int_{-\infty}^{\infty} \xi(x) \cdot \pi^{-1/2} (e^{-t}-1) x e^{x^{2}} dx,
\end{split}
\end{equation*}
i.e., $[V_{0}(t)]^{*}\nu$ is represented by the function $\pi^{-1/2} (e^{-t}-1) x e^{x^{2}}$. 
It is possible to perform the above calculation since we can, in the case of the OU-semigroup, calculate $[U_{\theta}(t)]^{*}$ and thus $[U_{\theta}(t)]^{*}\pi_{0}$ in closed form. Another possibility to calculate $\langle V_{0}(t)\xi | \nu\rangle$ would be to calculate $\nu$ and to use Remark \ref{integral-representation-of-V}. However to do this it is again necessary to calculate $[U_{0}(s)]^{*}$ in closed form. For the Wright--Fisher diffusion this has only been achieved in some special cases \cite{Kimura}, \cite{ChenStroock}.
\end{example}

\section{The Wright-Fisher diffusion}\label{section-WF}

We intend---in the case of the Wright--Fisher diffusion---to utilize Theorem \ref{main-result} (ii) for the series expansion of $\lim_{\theta\to 0} \theta^{-1}\langle [U_{\theta}(t)-U_{0}(t)] \xi | \pi_{0}\rangle$ via the series expansion of $V_{0}(t)$. This is done in the next section. 

\begin{remark} Let $\kappa>0$ be fixed throughout this section. For $\theta\geq 0$ we define operators $A_{\theta}:{\cal R}([0,1])\to {\cal R}([0,1])$ by
\begin{equation}\label{generator-of-WF}
A_{\theta}:=(1-x)\theta\frac{\partial }{\partial x} - x \kappa \frac{\partial}{\partial x}+x(1-x)\frac{\partial^2}{\partial x^2}.
\end{equation}
Let $\pi_{\theta}\in {\cal R}([0,1])'$ be implicitly defined by
\begin{equation}\label{pi_theta,mu}
\forall \xi \in {\cal R}([0,1])\ \ \ \ \langle \xi | \pi_{\theta}\rangle := \int \xi(x) \frac{\Gamma(\theta+\kappa )}{\Gamma(\theta)\Gamma(\kappa )} x^{\theta-1} (1-x)^{\kappa -1}\; dx
\end{equation}
for $\theta>0$ and by $\langle \xi | \pi_{0}\rangle := \xi(0)$. Then $A_{\theta}^{*}\pi_{\theta}=0$, i.e., 
\begin{equation}\label{beta-is-stationary}
\forall \xi\in {\cal R}([0,1])\ \ \ \langle A_{\theta} \xi | \pi_{\theta} \rangle =0.
\end{equation}
This is almost trivial in the case that $\theta=0$. For $\theta >0$ and polynomials $\xi$ of the form 
\begin{equation*}
\xi(x)=x^{2}(1-x)^{2}\cdot p(x)
\end{equation*}
---with $p(x)$ an arbitrary polynomial---we obtain (\ref{beta-is-stationary}) by partial integration
\begin{equation*}
\begin{split}
& \langle A_{\theta} \xi | \pi_{\theta} \rangle =\\
& \int \xi(x) \left [ -\frac{\partial}{\partial x}x^{\theta-1}(1-x)^{\kappa}\theta + \frac{\partial}{\partial x}x^{\theta}(1-x)^{\kappa-1}\kappa + \frac{\partial^{2}}{\partial x^{2}}x^{\theta}(1-x)^{\kappa} \right ]\; dx = 0,
\end{split}
\end{equation*}
which can be further extended to arbitrary polynomials by approximation arguments. (See also \cite[Chapter 4]{Ewens}) 
\end{remark}

\begin{definition}
Let $[U_{\theta}(t)]^{*}$ be the adjoint of $U_{\theta}(t)=e^{tA_{\theta}}$ with $A_{\theta}$ given by (\ref{generator-of-WF}). We call the semigroups $t\mapsto [U_{\theta}(t)]^{*}$  Wright--Fisher diffusions. 
\end{definition}
 
\begin{remark}
Since for $\theta>0$ the function $x\mapsto \frac{\Gamma(\theta+\kappa )}{\Gamma(\theta)\Gamma(\kappa )} x^{\theta-1} (1-x)^{\kappa -1}$ defined on $[0,1]$ is the density of a Beta distribution, we obtain that $\langle 1 | \pi_{\theta}\rangle = 1$ and thus further that $\theta^{-1}\langle 1 | \pi_{\theta}-\pi_{0}\rangle = 0$.
\end{remark}

\begin{remark}\label{d_pi_theta-FW}
For $n\geq 1$ we obtain from (\ref{pi_theta,mu}) that $\theta>0$ implies 
\begin{equation}
\langle x^{n} | \pi_{\theta}\rangle = \frac{\Gamma(\theta+n) \Gamma(\theta+\kappa)}{\Gamma(\theta) \Gamma(\theta+\kappa+n)}\cdot \langle 1 | \pi_{\theta+n} \rangle = \prod_{i=0}^{n-1} \frac{\theta+i}{\theta+\kappa+i}, 
\end{equation}
while $\langle x^{n} | \pi_{0}\rangle = 0^{n} = 0$. Thus 
\begin{equation*}
\begin{split}
\lim_{\theta\to 0} \theta^{-1}\langle x^{n} | \pi_{\theta}-\pi_{0}\rangle =  \lim_{\theta\to 0} \theta^ {-1}\langle x^{n} | \pi_{\theta}\rangle
= \lim_{\theta\to 0} \frac{1}{\theta} \prod_{i=0}^{n-1} \frac{\theta+i}{\theta+\kappa+i} = \frac{\Gamma(n)\Gamma(\kappa)}{\Gamma(\kappa+n)}.
\end{split}
\end{equation*}
\end{remark}

\begin{proposition}\label{A'_0-pi_0-FW}
Let $A_{\theta}$ be given by (\ref{generator-of-WF}). Then
\begin{equation}
\frac{d}{d\theta} \langle \xi | A_{\theta}^{*}\pi_{0}\rangle |_{\theta=0} = \lim_{\theta\to 0} \langle [A_{\theta}-A_{0}]\xi | \pi_{0}\rangle = \left \langle (1-x)\frac{\partial }{\partial x} \xi(x) | \pi_{0} \right \rangle = \xi'(0). 
\end{equation} 
\end{proposition}

\begin{remark}\label{|nu>=partial}
From Proposition \ref{A'_0-pi_0-FW}, Lemma \ref{A-fundamental-diff-relation-lemma} and Remark \ref{remark-diff-by-cont} we obtain that for $A_{\theta}$ given by (\ref{generator-of-WF})
\begin{equation}\label{=partial}
\lim_{\theta\to 0}\langle A_{0} \xi | \theta^{-1}(\pi_{\theta} - \pi_{0}) \rangle = \xi'(0)\ \ \hbox{  i.e. }\  \ \nu:=\lim_{\theta\to 0}A_{0}^{*} \theta^{-1} (\pi_{\theta} - \pi_{0}) = \frac{\partial}{\partial x}|_{x=0}
\end{equation}
and thus further from Remark \ref{main-cor} that
\begin{equation}\label{diff_U_for_WF}
\left \langle \xi \left | \frac{\partial [U_{\theta}(t)]^{*} \pi_{0} }{\partial \theta}|_{\theta=0}\right . \right \rangle = \langle V_{0}(t)\xi | \nu_{} \rangle = \frac{\partial}{\partial x}[V_{0}(t)\xi](x)|_{x=0}\
\end{equation}
\end{remark}

\begin{remark}\label{interpretation-of-WF} Calling an element $\mu\in {\cal R}({\mathbb R})'$ a probability-distribution if $\langle 1 | \mu \rangle =1$ and $\langle \xi | \mu \rangle \geq 0$ for all $\xi\geq 0$, we obtain the following interpretation of our Wright--Fisher diffusions $t\mapsto [U_{\theta}(t)]^{*}$:

Suppose that we start at time $0$ in a probability-distribution $\mu$ on $[0,1]$ giving us the proportion of individuals---in a large haploid population---that carries an allele A. Suppose further that we interpret the parameter $\kappa$ as the mutation rate at which allele A transforms into another allele B and $\theta$ as the mutation rate at which allele B transforms back into A. Then the probability-distribution $[U_{\theta}(t)]^{*}\mu$ gives us the proportion of individuals carrying allele A at time $t$.
Further the $n$-th moment $\langle x^{n} | [U_{\theta}(t)]^{*} \mu \rangle$ of $[U_{\theta}(t)]^{*} \mu$ gives us the probability that $n$ individuals independently chosen from the population at time $t$ all carry allele A.
The probability-distributions $\pi_{\theta}$ are the equilibrium distributions for the respective mutation rates. In the case that $\theta=0$ and $\kappa>0$ none of the individuals carries allele A in the equilibrium $\pi_{0}$. Suppose now that we start in the equilibrium $\pi_{0}$, but that the mutation rate $\theta$ is greater than $0$. Then the probability-distribution describing the population evolves according to $t\mapsto [U_{\theta}(t)]^{*}\pi_{0}$, and $t\mapsto \langle x^{n} | [U_{\theta}(t)]^{*}\pi_{0}\rangle$ gives us the evolution of the probability that $n$ individuals chosen at random from the population all carry allele A. An approximation of the probability $\langle x^{n} | [U_{\theta}(t)]^{*}\pi_{0}\rangle$ for fixed $t$ and small values of $\theta$ can be obtained by the first order expansion 
\begin{equation}\label{approximation}
\langle x^{n} | [U_{\theta}(t)]^{*}\pi_{0}\rangle \approx \langle x^{n} | [U_{0}(t)]^{*}\pi_{0}\rangle + \theta\cdot \left \langle x^{n} \left | \frac{\partial}{\partial \theta}[U_{\theta}(t)]^{*} \pi_{0} \right . \right \rangle.
\end{equation}
Since $\langle x^{n} | [U_{0}(t)]^{*}\pi_{0}\rangle=0$ for $n\geq 1$ and $\langle x^{n} | [U_{0}(t)]^{*}\pi_{0}\rangle=1$ for $n=0$, it suffices by (\ref{diff_U_for_WF}) to calculate
\begin{equation}
\frac{\partial}{\partial x}[V_{0}(t) x^{n}]|_{x=0} = \left \langle x^{n} \left | \frac{\partial}{\partial \theta}[U_{\theta}(t)]^{*} \pi_{0} \right . \right \rangle
\end{equation}
to determine the approximation (\ref{approximation}).
This is done for $n=0,1,2$ in the following example.
\end{remark}

\begin{example}
We calculate the derivative of the $0^{th}$, $1^{st}$ and $2^{nd}$ moments of $\theta\mapsto [U_{\theta}(t)]^{*}\pi_{0}$ at $\theta=0$ by calculating $\frac{\partial}{\partial x}[V_{0}(t)x^{i}]|_{x=0}$ for $i\in { 0, 1, 2}$.\\ \\
Clearly 
\begin{equation*}
A_{0}1=0, A_{0}x =-\kappa x\ \hbox{ and }\ A_{0}x^{2}=(-2\kappa-2)x^{2}+2x
\end{equation*}
From this we obtain for $k\geq 1$ (by induction on $k$) that $A_{0}^{k}1=0$, $A_{0}^{k}x=(-\kappa)^{k}x$ and 
\begin{equation*}
A_{0}^{k}x^{2}=(-2\kappa-2)^{k}x^{2}+2\cdot \sum_{i=0}^{k-1}(-2\kappa-2)^{k-(i+1)}(-\kappa)^{i}x
\end{equation*}
and thus further (note that $A^{0}=id$) that
\begin{equation}\label{0th}
[V_{0}(t)]1 = \sum_{k=1}^{\infty} \left [\frac{t^{k}}{k!}A^{k-1}\right ]1 = t,
\end{equation}
\begin{equation}\label{1st}
[V_{0}(t)]x = \sum_{k=1}^{\infty}\left [\frac{t^{k}}{k!}A^{k-1}\right ]x = \sum_{k=1}^{\infty} \frac{t^{k}}{k!}(-\kappa)^{k-1}x = \frac{e^{-\kappa t}}{-\kappa}x
\end{equation}
and
\begin{equation}\label{2nd}
\begin{split}
[V_{0}(t)]x^{2} & =  \sum_{k=1}^{\infty}\left [\frac{t^{k}}{k!}A^{k-1}\right ]x^{2}\\
& = \sum_{k=1}^{\infty} \frac{t^{k}}{k!} \left [(-2\kappa-2)^{k}x^{2}+2 \sum_{i=0}^{k-1}(-2\kappa-2)^{k-(i+1)}(-\kappa)^{i}x\right ]\\ 
& = e^{(-2\kappa-2) t}x^{2}+2 \sum_{k=1}^{\infty} \frac{t^{k}}{k!} (-2\kappa-2)^{k-1} \sum_{i=0}^{k-1}\left ( \frac{\kappa}{2\kappa+2}\right )^{i}x\\
& = e^{(-2\kappa-2) t}x^{2}+2 \sum_{k=1}^{\infty} \frac{t^{k}}{k!} \frac{(-2\kappa-2)^{k} (1- (\kappa/(2\kappa+2))^{k})}{(-2\kappa-2)(1- (\kappa/(2\kappa+2))} x\\
& = e^{(-2\kappa-2) t}x^{2}-\frac{2}{\kappa+2}  e^{-\kappa t} \left (e^{-(\kappa+2) t} - 1 \right ) x.
\end{split}
\end{equation}
From (\ref{0th}), (\ref{1st}) and (\ref{2nd}), we obtain that
\begin{equation}
\begin{split}
\frac{\partial}{\partial x}[V_{0}(t)]1=0,\ \ \frac{\partial}{\partial x}[V_{0}(t)]x = \frac{e^{-\kappa t}}{-\kappa}\ \hbox{ and }\\
\frac{\partial}{\partial x}[V_{0}(t)]x^{2}=2e^{(-2\kappa-2) t}x -\frac{2}{\kappa+2}  e^{-\kappa t} \left (e^{-(\kappa+2) t} - 1 \right )
\end{split}
\end{equation}
\end{example}

\begin{remark}\label{basis-remark}
Of course we can also calculate the derivatives of higher moments of $\theta\mapsto [U_{\theta}(t)]^{*}\pi_{0}$ at $\theta=0$ with increasing computational effort. Moreover there exists a basis of the space of polynomials---consisting of the vectors $1$, $x$ and the vectors $\xi_{n}$ defined in (\ref{def-xi_n}) below---for that a simple recursion for the calculation of $\lim_{\theta\to 0} \langle \theta^{-1} [U_{\theta}(t)-U_{0}(t)]\xi_{n} | \pi_{0} \rangle$ can be given. 
\end{remark}

\begin{theorem}\label{wf-rec-theorem}
Let $b_{n,0},\kappa \in {\mathbb R}$ be arbitrary, let $\gamma_{n,n}=1$ and let for $n\geq 2$, $k\geq 1$ and $2\leq m\leq n-1$
\begin{equation}\label{def-gamma_}
\gamma_{n,m-1}:=\frac{m(m-1)}{n(-\kappa-n+1)-(m-1)(-\kappa-m+2)}\cdot \gamma_{n,m},
\end{equation}
\begin{equation}\label{def-xi_n}
\xi_{n}:= \sum_{m=2}^{n} x^{m} \gamma_{n,m}
\end{equation}
and
\begin{equation}\label{def-b_}
b_{n,k}:=-\kappa b_{n,k-1} + (n(-\kappa-n+1))^{k-1}\cdot 2\gamma_{n,2}
\end{equation}
Then for $n\geq 2$ and $k\geq 1$
\begin{equation}\label{A^ĸ_psi_n}
A^{k}_{0}[\xi_{n}+b_{n,0}x+a] = (n(-\kappa-n+1))^ {k}\xi_{n} + b_{n,k} x
\end{equation}
\begin{equation}\label{<A^ĸ_psi_n|nu>}
\langle A^{k}_{0}[\xi_{n}+b_{n,0}x+a ] | \nu_{} \rangle = \frac{\partial}{\partial x} [A^{k}\xi_{n}](x)|_{x=0} = b_{n,k}
\end{equation}
and for $b_{n,0}=0$
\begin{equation}\label{recursion-formula-for-derivative}
\left \langle \xi_{n} \left | \frac{\partial [U_{\theta}(t)]^{*} \pi_{0}}{\partial \theta}|_{\theta=0} \right . \right \rangle = \lim_{\theta\to 0} \langle \theta^{-1} [U_{\theta}(t)-U_{0}(t)]\xi_{n} | \pi_{0} \rangle =  \sum_{k=1}^{\infty} \frac{t^{k}}{k!} b_{n,k-1}.
\end{equation}
\end{theorem}
\begin{proof}
To prove (\ref{A^ĸ_psi_n}) we show for $\beta=(n(-\kappa-n+1))^{k-1}$ and $i\geq 1$ that
\begin{equation}\label{A^1_psi_n}
A_{0}[ \beta\xi_{n}+b_{n,i-1}x+a] = n(-\kappa-n+1)\beta\xi_{n} + b_{n,i} x.
\end{equation}
This is done by a comparison of coefficients. The coefficient $\alpha_{n}$ of $x^{n}$ in $A_{0}[\beta\xi_{n}+b_{n,k-1}x+a]$ equals $n(-\kappa-n+1)\beta$ as the following calculation shows:
\begin{equation}\label{coefficient-x^{n}}
\alpha_{n}x^{n} = - x^{2} \frac{\partial^2}{\partial x^2} \beta x^{n} - \kappa x \frac{\partial }{\partial x} \beta x^{n} = n(-\kappa-n+1) \beta x^{n}.
\end{equation}
For $3\leq m\leq n$ the coefficient $\alpha_{m-1}$ of $x^{m-1}$ in $A_{0}[\beta\xi_{n}+b_{0} x+a]$ is given by $n(-\kappa-n+1)\beta \gamma_{n,m-1}$ as can be seen by the following calculation:
\begin{equation}\label{coefficient-general-m}
\begin{split}
\alpha_{m-1} x^{m-1} & = x\frac{\partial^2}{\partial x^2} x^{m} \beta \gamma_{n,m}  - x^{2} \frac{\partial^2}{\partial x^2} x^{m-1}  \beta \gamma_{n,m-1} - \kappa x \frac{\partial }{\partial x} x^{m-1} \beta \gamma_{n,m-1}\\ 
& = (m(m-1)\beta\gamma_{n,m} + (m-1)(-\kappa-m+2)\beta \gamma_{n,m-1}) x^{m-1} \\
& =  n(-\kappa-n+1)\beta \gamma_{n,m-1}\, x^{m-1},
\end{split}
\end{equation}
with the last equality a consequence of (\ref{def-gamma_}).
Thus it remains to calculate the coefficients $\alpha_{1}$ of $x^{1}$ and $\alpha_{0}$ of $x^{0}=1$. We obtain $\alpha_{0}=0$ since the first and second order derivatives applied to the constant function gives $0$. We further obtain $\alpha_{1}=-\kappa b_{n,i-1} + \beta 2\gamma_{n,2}= b_{n,i}$ by the following calculation:
\begin{equation}\label{coefficient-x^1}
\begin{split}
\alpha_{1}x & = x \frac{\partial^2}{\partial x^2} x^{2} \beta \gamma_{n,2}  - x^{2} \frac{\partial^2}{\partial x^2} x \beta b_{n,i-1} - \kappa x \frac{\partial }{\partial x} x \beta b_{n,i-1}\\
& =  (\beta\cdot 2\gamma_{n,2} x + 0 -\kappa b_{n,i-1})x=b_{n,i}x,
\end{split}
\end{equation}
with the last equality a consequence of (\ref{def-b_})
From (\ref{coefficient-x^{n}}), (\ref{coefficient-general-m}) and (\ref{coefficient-x^1}) we obtain (\ref{A^1_psi_n}) and by recursion over (\ref{A^1_psi_n}) (with $i$ ranging from $1$ to $k$) we obtain (\ref{A^ĸ_psi_n}). Equation (\ref{<A^ĸ_psi_n|nu>}) is a consequence of (\ref{A^ĸ_psi_n}) and Remark \ref{|nu>=partial}.\\ \\
Finally we obtain (\ref{recursion-formula-for-derivative}) by the following calculation

\begin{equation}
\begin{split}
\lim_{\theta\to 0} \langle \theta^{-1} [U_{\theta}(t)-U_{0}(t)]\xi_{n} | \pi_{0} \rangle & \stackrel{(a)}{=} \langle V_{0}(t)  \xi_{n} | \nu_{} \rangle\\
& \stackrel{(b)}{=} \lim_{N\to\infty} \left \langle \left . \sum_{k=1}^{N} \frac{t}{k} \frac{(t A_{0})^{k-1}}{(k-1)!}  \xi_{n} \right | \nu_{}\right \rangle
\stackrel{(c)}{=} \sum_{k=1}^{\infty} \frac{t^{k}}{k!} b_{n,k-1},
\end{split}
\end{equation}
with (a), (b) and (c) consequences of Theorem \ref{main-result} (ii), Lemma \ref{convergence-of-exp} and equation (\ref{<A^ĸ_psi_n|nu>}) with $b_{n,0}=a=0$, respectively. 
\end{proof}

\end{document}